\documentclass[12 pt,twoside]{amsart}

\usepackage{amsopn}
\usepackage{amssymb}
\usepackage{amscd}

\newtheorem{theorem}{Theorem}[section]
\newtheorem{lemma}[theorem]{Lemma}
\newtheorem{proposition}[theorem]{Proposition}
\newtheorem{corollary}[theorem]{Corollary}

\theoremstyle{definition}
\newtheorem{definition}[theorem]{Definition}

\theoremstyle{remark}
\newtheorem{remark}[theorem]{Remark}

\numberwithin{equation}{section}

\begin{document}

\title[$J$-class operators and hypercyclicity]{$J$-class operators and hypercyclicity}

\author[George Costakis]{George Costakis}
\address{Department of Mathematics, University of Crete, Knossos Avenue, GR-714 09 Heraklion, Crete, Greece}
\email{costakis@math.uoc.gr}
\thanks{}

\author[Antonios Manoussos]{Antonios Manoussos}
\address{Fakult\"{a}t f\"{u}r Mathematik, SFB 701, Universit\"{a}t Bielefeld, Postfach 100131, D-33501 Bielefeld, Germany}
\email{amanouss@math.uni-bielefeld.de}
\thanks{During this research the second author was fully supported by SFB 701 ``Spektrale Strukturen und
Topologische Methoden in der Mathematik" at the University of Bielefeld, Germany. He would also like to express his gratitude to Professor H. Abels for his support.}


\date{}


\begin{abstract}
The purpose of the present work is to treat a new notion related to linear dynamics, which can be viewed as a ``localization" of the notion of hypercyclicity. In
particular, let $T$ be a bounded linear operator acting on a Banach space $X$ and let $x$ be a non-zero vector in $X$ such that for every open neighborhood $U\subset
X$ of $x$ and every non-empty open set $V\subset X$ there exists a positive integer $n$ such that $T^{n}U\cap V\neq\emptyset$. In this case $T$ will be called a
$J$-class operator. We investigate the class of operators satisfying the above property and provide various examples. It is worthwhile to mention that many results
from the theory of hypercyclic operators have their analogues in this setting. For example we establish results related to the Bourdon-Feldman theorem and we
characterize the $J$-class weighted shifts. We would also like to stress that even non-separable Banach spaces which do not support topologically transitive
operators, as for example $l^{\infty}(\mathbb{N})$, do admit $J$-class operators.
\end{abstract}

\maketitle

\section{Introduction}
Let $X$ be a complex (or real) Banach space. In the rest of the paper the symbol $T$ stands for a bounded linear operator acting on $X$. We first fix some notation.
Consider any subset $C$ of $X$. The symbols $C^o$, $\overline{C}$ and $\partial C$ denote the interior, the closure and the boundary of $C$ respectively. The symbol
$Orb(T,C)$ denotes the orbit of $C$ under $T$, i.e. $Orb(T,C)=\{ T^nx: x\in C,\,\, n=0,1,2,\ldots \}$. If $C=\{ x\}$ is a singleton and the orbit $Orb(T,x)$ is dense
in $X$, the operator $T$ is called hypercyclic and the vector $x$ is a hypercyclic vector for $T$. If $C=\{ \lambda x: \lambda \in \mathbb{C} \}=\mathbb{C} x$ and the
set $Orb(T,C)$ is dense in $X$, the operator $T$ is called supercyclic and the vector $x$ is a supercyclic vector for $T$. A nice source of examples and properties of
hypercyclic and supercyclic operators is the survey article \cite{GE}, see also some recent survey articles \cite{Sh}, \cite{GE2}, \cite{MoSa2}, \cite{BoMaPe},
\cite{Fe3}, \cite{GE3} and the recent book \cite{BaMa3}. Observe that in case the operator $T$ is hypercyclic the underlying Banach space $X$ should be separable.
Then it is well known and easy to show that \textit{an operator $T:X\to X$ is hypercyclic if and only if for every pair of non-empty open sets $U,V$ of $X$ there
exists a positive integer $n$ such that $T^n(U)\cap V\neq \emptyset$}. The purpose of this paper is twofold. Firstly we somehow ``localize" the notion of
hypercyclicity by introducing certain sets, which we call $J$-sets. The notion of $J$-sets is well known in the theory of topological dynamics, see \cite{BaSz}.
Roughly speaking, if $x$ is a vector in $X$ and $T$ an operator, then the corresponding $J$-set of $x$ under $T$ describes the asymptotic behavior of all vectors
nearby $x$. To be precise for a given vector $x\in X$ we define

\begin{equation*}
\begin{split}
J(x)=\{ &y\in X:\,\mbox{ there exist a strictly increasing sequence of positive}\\
&\mbox{integers}\,\{k_{n}\}\,\mbox{and a sequence }\,\{x_{n}\}\subset X\,\mbox{such that}\, x_{n}\rightarrow x\,
\mbox{and}\\
&T^{k_{n}}x_{n}\rightarrow y\}.
\end{split}
\end{equation*}

Secondly we try to develop a systematic study of operators whose $J$-set under some vector is the whole space. As it turns out this new class of operators although is
different from the class of hypercyclic operators, shares some similarities with the behavior of hypercyclic operators. In fact it is not difficult to see that if $T$
is hypercyclic then $J(x)=X$ for every $x\in X$. On the other hand we provide examples of operators $T$ such that $J(x)=X$ for some vector $x\in X$ but $T$ fails to
be hypercyclic and in general $T$ need not be even multi-cyclic. This should be compared with the results of Feldman in \cite{Fel1} where he shows that a countably
hypercyclic operator need not be multi-cyclic. We would like to stress that non-separable Banach spaces, as the space $l^{\infty}(\mathbb{N})$ of bounded sequences,
supports $J$-class operators, (see Proposition 5.2), while it is known that the space $l^{\infty}(\mathbb{N})$ does not support topologically transitive operators,
see \cite{BeKa}.

The paper is organized as follows. In section 2 we define the $J$-sets and we examine some basic properties of these sets. In section 3 we investigate the relation
between hypercyclicity and $J$-sets. In particular we show that $T:X\to X$ is hypercyclic if and only if there exists a cyclic vector $x\in X$ such that $J(x)=X$.
Recall that a vector $x$ is cyclic for $T$ if the linear span of the orbit $Orb(T,x)$ is dense in $X$. The main result of section 4 is a generalization of a theorem
due to Bourdon and Feldman, see \cite{BF}. Namely, we show that if $x$ is a cyclic vector for an operator $T:X\to X$ and the set $J(x)$ has non-empty interior then
$J(x)=X$ and, in addition, $T$ is hypercyclic. In section 5 we introduce the notion of $J$-class operator and we establish some of its properties. We also present
examples of $J$-class operators which are not hypercyclic. On the other hand, we show that if $T$ is a  bilateral or a unilateral weighted shift on the space of
square summable sequences then $T$ is hypercyclic if and only if $T$ is a $J$-class operator. Finally, in section 6 we give a list of open problems.

\section{Preliminaries and basic notions}
If one wants to work on general non-separable Banach spaces and in order to investigate the dynamical behavior of the iterates of $T$, the suitable substitute of
hypercyclicity is the following well known notion of topological transitivity which is frequently used in dynamical systems.

\begin{definition}
An operator $T:X\to X$ is called \textit{topologically transitive} if for every pair of open sets $U,V$ of $X$ there exists a positive integer $n$ such that $T^nU\cap
V\neq \emptyset $.
\end{definition}

\begin{definition}
Let $T:X\rightarrow X$ be an operator. For every $x\in X$ the sets
\begin{equation*}
\begin{split}
L(x)=\{&y\in X:\,\mbox{ there exists a strictly increasing sequence}\\
       &\mbox{of positive integers}\,\,\{k_{n}\}\mbox{ such that }\,T^{k_{n} }x\rightarrow y\}
\end{split}
\end{equation*}
and
\begin{equation*}
\begin{split}
J(x)=\{ &y\in X:\,\mbox{ there exist a strictly increasing sequence of positive}\\
&\mbox{integers}\,\{k_{n}\}\,\mbox{and a sequence }\,\{x_{n}\}\subset X\,\mbox{such that}\, x_{n}\rightarrow x\,
\mbox{and}\\
&T^{k_{n}}x_{n}\rightarrow y\}
\end{split}
\end{equation*}
denote the \textit{limit set} and the \textit{extended (prolongational) limit set} of $x$ under $T$ respectively. In case $T$ is invertible and for every $x\in X$ the
sets $L^{+}(x)$, $J^{+}(x)$ ($L^{-}(x)$, $J^{-}(x)$) denote the limit set and the extended limit set of $x$ under $T$ ($T^{-1}$).
\end{definition}

\begin{remark}
An equivalent definition of $J(x)$ is the following.
\begin{equation*}
\begin{split}
J(x)=\{ &y\in X:\,\mbox{for every pair of neighborhoods}\, U,V\,\mbox{of}\, x,y\\
        &\mbox{respectively, there exists a positive integer}\,n,\\
        &\mbox{such that}\, T^{n}U\cap V\neq\emptyset \}.
\end{split}
\end{equation*}
Observe now that $T$ is topologically transitive if and only if $J(x)=X$ for every $x\in X$.
\end{remark}

\begin{definition}
Let $T:X\rightarrow X$ be an operator. A vector $x$ is called \textit{periodic} for $T$ if there exists a positive integer $n$ such that $T^nx=x$.
\end{definition}

The proof of the following lemma can be found in \cite{CosMa}.

\begin{lemma}
Let $T:X\rightarrow X$ be an operator and $\{ x_{n}\}$, $\{ y_{n} \}$ be two sequences in $X$ such that
$x_{n}\rightarrow x$ and $y_{n} \rightarrow y$ for some $x,y\in X$. If $y_{n} \in J(x_{n})$ for every
$n=1,2,\ldots$, then $y\in J(x)$.
\end{lemma}

\begin{proposition}
For all $x\in X$ the sets $L(x)$, $J(x)$ are closed and $T$-invariant.
\end{proposition}
\begin{proof}
It is an immediate consequence of the previous lemma.
\end{proof}

\begin{remark}
Note that the set $J(x)$ is not always invariant under the operation $T^{-1}$ even in the case $T$ is surjective. For example consider the operator $T=\frac{1}{2}B$
where $B$ is the backward shift operator on, $l^{2}(\mathbb{N})$, the space of square summable sequences. Since $\| T\|=\frac{1}{2}$ it follows that $L(x)=J(x)=\{
0\}$ for every $x\in l^{2}(\mathbb{N})$. For any non-zero vector $y\in KerT$ we have $Ty=0\in J(x)$ and $y\in X\setminus J(x)$. However, if $T$ is invertible it is
easy to verify the following.
\end{remark}
\begin{proposition}
Let $T:X\rightarrow X$ be an invertible operator. Then $T^{-1}J(x)=J(x)$ for every $x\in X$.
\end{proposition}
\begin{proof}
By Proposition 2.6 it follows that $J(x)\subset T^{-1}J(x)$. Take $y\in T^{-1}J(x)$. There are a strictly increasing sequence $\{ k_{n}\}$ of positive integers and a
sequence $\{ x_n\}\subset X$ so that $x_{n}\rightarrow x$ and $T^{k_{n}}x_{n}\rightarrow Ty$, hence $T^{k_{n}-1}x_{n}\rightarrow y$.
\end{proof}

\begin{proposition}
Let $T:X\rightarrow X$ be an operator and $x,y\in X$. Then $y\in J^{+}(x)\,\,\mbox{if and only if}\,\, x\in
J^{-}(y)$.
\end{proposition}
\begin{proof}
If $y\in J^{+}(x)$ there exist a strictly increasing sequence $\{ k_{n}\}$ of positive integers and a sequence $\{ x_n\}\subset X$ such that $x_{n}\rightarrow x$ and
$T^{k_{n}}x_{n}\rightarrow y$. Then $T^{-k_{n}}(T^{k_{n}}x_{n})=x_n\rightarrow x$, hence $x\in J^{-}(y)$.
\end{proof}

\begin{proposition}
Let $T:X\rightarrow X$ be an operator. If $T$ is power bounded then  $J(x)=L(x)$ for every $x\in X$.
\end{proposition}
\begin{proof}
Since $T$ is power bounded there exists a positive number $M$ such that $\| T^n\| \leq M$ for every positive
integer $n$. Fix a vector $x\in X$. If $J(x)=\emptyset $ there is nothing to prove. Therefore assume that
$J(x)\neq \emptyset$. Since the inclusion $L(x)\subset J(x)$ is always true, it suffices to show that
$J(x)\subset L(x)$. Take $y\in J(x)$. There exist a strictly increasing sequence $\{ k_{n}\}$ of positive
integers and a sequence $\{ x_n\}\subset X$ such that $x_{n}\rightarrow x$ and $T^{k_{n}}x_{n}\rightarrow y$.
Then we have $\| T^{k_n}x-y\| \leq \| T^{k_n}x-T^{k_n}x_n\| +\| T^{k_n}x_n-y\| \leq M\| x-x_n\| +\|
T^{k_n}x_n-y\|$  and letting $n$ goes to infinity to the above inequality, we get that $y\in L(x)$.
\end{proof}

\begin{lemma}
Let $T:X\rightarrow X$ be an operator. If $J(x)=X$ for some non-zero vector $x\in X$ then $J(\lambda x)=X$ for
every $\lambda \in \mathbb{C}$.

\end{lemma}
\begin{proof}
For $\lambda \in \mathbb{C}\setminus \{ 0\}$ it is easy to see that $J(\lambda x)=X$. It remains to show that $J(0)=X$. Fix a sequence of non-zero complex numbers $\{
\lambda_n \}$ converging to $0$ and take $y\in J(x)$. Then $y\in J(\lambda_nx)$ for every $n$ and since $\lambda_n\to 0$, Lemma 2.5 implies that $y\in J(0)$. Hence
$J(0)=X$.
\end{proof}

\begin{proposition}
Let $T:X\rightarrow X$ be an operator. Define the set $A=\{ x\in X: J(x)=X\}$. Then $A$ is a closed, connected
and $T(A)\subset A$.
\end{proposition}
\begin{proof}
The $T$-invariance follows immediately from the $T$-invariance of $J(x)$. By Lemma 2.5 we conclude that $A$ is closed. Let $x\in A$. Lemma 2.11 implies that for every
$\lambda \in \mathbb{C}$, $J(0)=J(\lambda x)=X$, hence $A$ is connected.
\end{proof}

\section{A characterization of hypercyclic operators through $J$-sets}

The following characterization of hypercyclic operators appears more or less in \cite{GE}. However we sketch the
proof for the purpose of completeness.

\begin{theorem}
Let $T:X\rightarrow X$ be an operator acting on a separable Banach space $X$. The following are equivalent.
\begin{enumerate}
\item[(i)]    $T$ is hypercyclic;

\item[(ii)]   For every $x\in X$ it holds that $J(x)=X$;

\item[(iii)]  The set $A=\{ x\in X: J(x)=X\}$ is dense in $X$;

\item[(iv)]   The set $A=\{ x\in X: J(x)=X\}$ has non-empty interior.
\end{enumerate}
\end{theorem}
\begin{proof}
We first prove that (i) implies (ii). Let $x,y\in X$. Since the set of hypercyclic vectors is $G_{\delta}$ and dense in $X$ there exist a sequence $\{ x_{n}\}$ of
hypercyclic vectors and a strictly increasing sequence $\{ k_{n}\}$ of positive integers such that $x_{n}\rightarrow x$ and $T^{k_{n}}x_{n}\rightarrow y$ as $n\to
\infty$. Hence $y\in J(x)$. That (ii) implies (iii) is trivial. A consequence of Lemma 2.5 is that (iii) gives (ii). Next we show that (iv) implies (ii). Fix $x\in
A^{o}$ and consider $y\in X$ arbitrary. Then $y\in J(x)=X$, hence there exist a sequence $\{ x_{n}\}\subset X$ and a strictly increasing sequence $\{ k_{n}\}$ of
positive integers such that $x_{n} \rightarrow x$ and $T^{k_{n}}x_{n}\rightarrow y$. Since $x\in A^{o}$ without loss of generality we may assume that $x_{n}\in A$ for
every $n$. Moreover $A$ is $T$-invariant, hence $T^{k_{n}}x_{n}\in A$ for every $n$. Since $T^{k_{n}}x_{n}\rightarrow y$ and $A$ is closed we conclude that $y\in A$.
Let us now prove that (ii) implies (i). Fix $\{ x_{j} \}$ a countable dense set of $X$. Define the sets $ E(j,s,n)=\{ x\in X: \| T^{n}x-x_{j}\| <\frac{1}{s} \}$ for
every $j,s= 1,2,\dots$ and every $n=0,1,2,\dots$.  In view of Baire's Category Theorem and the well known set theoretical description of hypercyclic vectors through
the sets $E(j,s,n)$, it suffices to show that the set $\bigcup_{n=0}^{\infty}E(j,s,n)$ is dense in $X$ for every $j,s$. Indeed, let $y\in X$, $\epsilon>0$, $j,s$ be
given. Since $J(y)=X$, there exist $x\in X$ and $n\in \mathbb{N}$ such that $\| x-y\| <\epsilon$ and $\| T^{n}x-x_{j} \| <1/s$.
\end{proof}

The following lemma -see also Corollary 3.4- which is of great importance in the present paper, gives information about the spectrum of the adjoint $T^*$ of an
operator $T:X\to X$ provided there is a vector $x\in X$ whose extended limit set $J(x)$ has non-empty interior. The corresponding result for hypercyclic operators has
been proven by P. Bourdon in \cite{Bourdon}.

\begin{lemma}
Let $T:X\rightarrow X$ be an operator acting on a complex or real Banach space $X$. Suppose there exists a vector $x\in X$ such that $J(x)$ has non-empty interior and
$x $ is cyclic for $T$. Then for every non-zero polynomial $P$ the operator $P(T)$ has dense range. In particular the point spectrum $\sigma_p(T^*)$ of $T^*$ (the
adjoint operator of $T$) is empty, i.e. $\sigma_p(T^*)=\emptyset$.
\end{lemma}
\begin{proof}
Assume first that $X$ is a complex Banach space. Since $P(T)$ can be decomposed in the form $P(T)=\alpha(T-{\lambda}_{1} I)(T-{\lambda}_{2}I)\ldots(T-{\lambda}_{k}I)$
for some $\alpha,{\lambda}_{i} \in\mathbb{C}$, $i=1,\ldots,k$, where $I$ stands for the identity operator, it suffices to show that $T-\lambda I$ has dense range for
any $\lambda \in\mathbb{C}$. If not, there exists a non-zero linear functional $x^{*}$ such that $x^{*}((T-\lambda I)(x))=0$ for every $ x\in X$. The last implies
that $x^{*}(T^{n}x)={\lambda}^{n}x^{*}(x)$ for every $x\in X$ and every $n$ non-negative integer. Take $y$ in the interior of $J(x)$. Then there exist a sequence $\{
x_{n} \} \subset X$ and a strictly increasing sequence $\{ k_{n} \}$ of positive integers such that $x_{n}\rightarrow x$ and $T^{k_{n}} x_{n}\rightarrow y$ as
$n\rightarrow+\infty$. Suppose first that $|\lambda|<1$. Observe that $x^{\ast}(T^{k_{n}} x_{n})={\lambda}^{k_{n}}x^{\ast}(x_{n})$ and letting $n\rightarrow+\infty$
we arrive at $x^{\ast}(y)=0$. Since the functional $x^{\ast}$ is zero on an open subset of $X$ must be identically zero on $X$, which is a contradiction. Working for
$|\lambda|=1$ as before, it is easy to show that for every $y$ in the interior of $J(x)$, $x^{\ast}(y)=\mu x^{\ast}(x)$ for some $\mu\in\mathbb{C}$ with $|\mu|=1$,
which is again a contradiction since $x^{\ast}$ is surjective. Finally we deal with the case $|\lambda|>1$. \textit{At this part of the proof we shall use the
hypothesis that $x$ is cyclic}. Letting $n\rightarrow+\infty$ to the next relation, $x^{\ast}(x_{n})=\frac{1}{{\lambda}^{k_{n}}}x^{\ast}(T^{k_{n}}x_{n})$, it is plain
that $x^{\ast}(x)=0$ and therefore $x^{\ast}(T^{n}x)=0$ for every $n$ non-negative integer. The last implies that $x^{\ast}(P(T)x)=0$ for every $P$ non-zero
polynomial and since $x$ is cyclic the linear functional $x^{\ast}$ vanishes everywhere, which gives a contradiction. It remains to handle the real case. For that it
suffices to consider the case where $P$ is an irreducible and monic polynomial of the form $P(t)=t^2-2Re(w)t+|w|^2$ for some non-real complex number $w$. Assume that
$P(T)$ does not have dense range. Then there exists a non-zero $x^{\ast} \in Ker(P(T)^{\ast})$. Following the proof of the main result in \cite{Bes}, there exists a
real $2\times 2$ matrix $A$ such that $J_{A^t}((x^*(Tx), x^*(x))^t)=\mathbb{R}^2$, where the symbol $A^t$ stands for the transpose of $A$. By Proposition 5.5 (which
hold in the real case as well) we get $x^*(Tx)=x^*(x)=0$. The last implies that $x^*(Q(T)x)=0$ for every real polynomial $Q$. Since $x$ is cyclic we conclude that
$x^*=0$ which is a contradiction. This completes the proof of the lemma.
\end{proof}

\begin{theorem}
Let $T:X\rightarrow X$ be an operator acting on a separable Banach space $X$. Then $T$ is hypercyclic if and only if there exists a cyclic vector $x\in X$ for $T$
such that $J(x)=X$.
\end{theorem}
\begin{proof}
We need only to prove that if $x\in X$ is a cyclic vector for $T$ and $J(x)=X$ then $T$ is hypercyclic. Take any non-zero polynomial $P$. It is easy to check that
$P(T)(J(x))\subset J(P(T)x). $ By the previous lemma it follows that $P(T)$ has dense range and since $J(x)=X$ we conclude that $X=\overline{ P(T)(X) }\subset
J(P(T)x) $. Therefore $J(P(T)x)=X$ for every non-zero polynomial $P$. The fact that $x$ is a cyclic vector it now implies that there exists a dense set $D$ in $X$ so
that $J(y)=X$ for every $y\in D$. Hence, in view of Theorem 3.1, $T$ is hypercyclic.
\end{proof}

\begin{corollary}
Let $T:X\rightarrow X$ be an operator. Suppose there exists a vector $x\in X$ such that $J(x)$ has non-empty interior. Then for every $\lambda\in\mathbb{C}$ with $|
\lambda | \leq 1$ the operator $T-\lambda I$ has dense range.
\end{corollary}
\begin{proof}
See the proof of Lemma 3.2.
\end{proof}

\begin{remark}
At this point we would like to comment on Theorem 3.3. First of all under the hypothesis that $x$ is a cyclic vector for $T$ and $J(x)=X$ one cannot get a stronger
conclusion than $T$ is hypercyclic. In particular it is not true in general that $x$ is a hypercyclic vector. To see this, take $T=2B$ where $B$ is the backward shift
operator acting on the space of square summable sequences $l^2(\mathbb{N})$ over $\mathbb{C}$. In \cite{Fe2} Feldman showed that for a given positive number
$\epsilon$ there exists a vector $x\in l^2(\mathbb{N})$ such that the set $Orb(2B,x)$ is $\epsilon$-dense in $l^2(\mathbb{N})$ (this means that for every $y\in
l^2(\mathbb{N})$ there exists a positive integer $n$ such that $T^{n}x$ is $\epsilon$-close to $y$), but $x$ is not hypercyclic for $2B$. It is straightforward to
check that $x$ is supercyclic for $2B$ and hence it is cyclic. In addition $J(x)=l^2(\mathbb{N})$ since $2B$ is hypercyclic (see Theorem 3.1).
\end{remark}

\begin{remark}
Let us now show that the hypothesis $x$ is cyclic in Theorem 3.3 cannot be omitted. Let $B:l^2(\mathbb{N})\rightarrow l^2(\mathbb{N})$ be the backward shift operator.
Consider the operator $ T=2I\oplus 2B:\mathbb{C}\oplus l^2(\mathbb{N})\rightarrow \mathbb{C}\oplus l^2(\mathbb{N}), $ where $I$ is the identity operator acting on
$\mathbb{C}$. It is obvious that $2I\oplus 2B$ is not a hypercyclic operator. However we shall show that for every hypercyclic vector $y\in l^2(\mathbb{N})$ for $2B$
it holds that $J(0\oplus y)=\mathbb{C}\oplus l^2(\mathbb{N})$. Therefore there exist (non-cyclic) non-zero vectors $x\in \mathbb{C}\oplus l^2(\mathbb{N})$ with
$J(x)=\mathbb{C}\oplus l^2(\mathbb{N})$ and $T$ is not hypercyclic. Indeed, fix a hypercyclic vector $y\in l^2(\mathbb{N})$ for $2B$ and let $\lambda\in\mathbb{C}$,
$w\in l^2(\mathbb{N})$. There exists a strictly increasing sequence of positive integers $\{ k_n \}$ such that $T^{k_n}y\rightarrow w$. Define $
x_n=\frac{\lambda}{2^{k_n}}\oplus y $. Then $ x_n\rightarrow 0\oplus y$ and $T^{k_n}x_n\rightarrow \lambda\oplus w$. Hence, $J(0\oplus y)=\mathbb{C}\oplus
l^2(\mathbb{N})$.
\end{remark}

\section{An extension of Bourdon-Feldman's theorem}
In this section we establish an extension of the following striking result due to Bourdon and Feldman \cite{BF}: \textit{if $X$ is a separable Banach space, $T:X\to
X$ an operator and for some vector $x\in X$ the orbit $Orb(T,x)$ is somewhere dense then $\overline{Orb(T,x)}=X$.} This theorem was an answer to a question raised by
Peris in \cite{P}. We shall prove the following theorem.

\begin{theorem}
Let $x$ be a cyclic vector for $T$. If $J(x)^o\neq \emptyset$ then $J(x)=X$.
\end{theorem}

In order to prove Theorem 4.1 we follow the steps of the proof of Bourdon-Feldman's theorem. Of course there are some extra technicalities which have to be taken care
since the orbit $Orb(T,x)$ of $x$ under $T$ is replaced by the extended limit set $J(x)$ of $x$.

\begin{lemma}
If for some non-zero polynomial $P$ the operator $P(T)$ has dense range and $x$ is a cyclic vector for $T$ then $P(T)x$ is cyclic for $T$.
\end{lemma}
\begin{proof}
Take $P(T)y$ for some $y\in X$. Since $x$ is cyclic there is a sequence of polynomials $\{ Q_{n}\} $ such that
$Q_{n}(T)x\rightarrow y$. Therefore, $Q_{n}(T)(P(T)x)\rightarrow P(T)y$.
\end{proof}

\begin{lemma}
Assume that $x$ is a cyclic vector for $T$ and $J(x)$ has non-empty interior. Then the set $X\setminus J(x)^{o}$ is $T$-invariant.
\end{lemma}
\begin{proof}
We argue by contradiction. Let $y\in X\setminus J(x)^{o}$ be such that $Ty\in J(x)^{o}$. By the continuity of $T$ we may assume that $ y\notin J(x)$. Moreover, since
$x$ is cyclic we may find a non-zero polynomial $P(T)$ such that $ P(T)x\in X\setminus J(x)^{o}$ and $ TP(T)x\in J(x)^{o}$. Hence, there exist a sequence $\{ x_{n}\}
\subset X$ and a strictly increasing sequence of positive integers $\{ k_{n}\}$ such that $x_{n} \rightarrow x$ and $T^{k_{n} }x_{n} \rightarrow TP(T)x$. Taking any
polynomial $Q$ we get $Q(T)x_{n} \rightarrow Q(T)x $ and $T^{k_{n}}Q(T)x_{n}= Q(T)(T^{k_{n}}x_{n})\rightarrow Q(T)TP(T)x $. So it follows that $P(T)TQ(T)x\in
J(Q(T)x)$ for every polynomial $Q$. But $J(Q(T)x)\subset J(TQ(T)x)$, hence we get $P(T)TQ(T)x\in J(TQ(T)x)$ for every polynomial $Q$. By Lemmata 3.2 and 4.2, $Tx$ is
a cyclic vector for $T$, hence there exists a sequence of the form $\{ Q_{n}(T)x\}$, for some non-zero polynomials $Q_{n}$, such that $TQ_{n}(T)x\rightarrow x$.
Therefore it follows that $ P(T)TQ_{n}(T)x\rightarrow P(T)x$. Observe that $P(T)TQ_{n}(T)x\in J(TQ_{n}(T)x)$ and using Lemma 2.5 it follows that $P(T)x\in J(x)$ which
is a contradiction.
\end{proof}

\begin{lemma}
Assume that $x$ is a cyclic vector for $T$ and $J(x)$ has non-empty interior. Suppose that $Q(T)x\in X\setminus J(x)$ for some non-zero polynomial $Q$. Then
$Q(T)(J(x))\subset X\setminus J(x)^{o}$.
\end{lemma}
\begin{proof}
Let $y\in J(x)$. There exist a sequence $\{ x_{n}\} \subset X$ and a strictly increasing sequence of positive integers $\{ k_{n}\}$ such that $x_{n}\rightarrow x$ and
$T^{k_{n}}x_{n}\rightarrow y$. Since $X\setminus J(x)$ is an open set we may assume that $Q(T)x_{n}\in X\setminus J(x)$ for every $n$ and thus $Q(T)x_{n}\in
X\setminus J(x)^{o}$. By Lemma 4.3 the set $X\setminus J(x)^{o}$ is $T$-invariant, therefore $T^{k_{n}}Q(T)x=Q(T)T^{k_{n} }x_{n}\in X\setminus J(x)^{o}$. Now it is
plain that $Q(T)y\in X\setminus J(x)^{o}$.
\end{proof}

\begin{lemma}
Assume that $x$ is a cyclic vector for $T$, $J(x)$ has non-empty interior and let $P$ be any non zero polynomial. Then $P(T)x\notin\partial(J(x)^{o})$.
\end{lemma}
\begin{proof}
In view of Lemma 4.4 let us define the set
\begin{equation*}
{\mathcal A}=\{ Q:\, Q\, \mbox{is a polynomial and}\, Q(T)x\in X\setminus J(x)\}.
\end{equation*}
Note that the set $\{ Qx:\, Q\in {\mathcal A}\}$ is dense in  $X\setminus J(x)^{o}$. We argue by contradiction. Suppose there exists a non-zero polynomial $P$ so that
$P(T)x\in\partial(J(x)^{o})$. The inclusion $\partial(J(x)^{o})\subset\partial J(x)$ gives that $P(T)x\in\partial(X\setminus J(x))$. We will prove that
$P(T)(J(x)^{o})\subset X\setminus J(x)^{o}$. Since $x$ is a cyclic vector and $J(x)^{o}$ is open, it is enough to show that: if $S(T)x\in J(x)^{o}$ for some non-zero
polynomial $S$ then $P(T)S(T)x\in X\setminus J(x)^{o}$. We have $P(T)x\in\partial (X\setminus J(x))$. Therefore there exists a sequence $\{Q_{n}(T)x\}$ such that
$Q_{n}\in {\mathcal A}$ and $Q_{n}(T)x\rightarrow P(T)x$. Hence Lemma 4.4 yields that $Q_{n}(T)S(T)x\in X\setminus J(x)^{o}$. So, we get $ Q_{n}(T)S(T)x\rightarrow
P(T)S(T)x$ and $P(T)S(T)x\in X\setminus J(x)^{o}$. Consider the set $D:=J(x)^{o}\bigcup\{ Q(T)x\, :\, Q\in {\mathcal A}\}$ which is dense in $X$. By Lemma 3.2,
$P(T)D$ is dense in $X$. Since $P(T)x\in J(x)$, Lemma 4.4 implies that $Q(T)P(T)x\in X\setminus J(x)^o$ for every  $Q\in {\mathcal A}$. Hence
\begin{equation*}
P(T)D=P(T)(J(x)^{o})\bigcup\{ P(T)Q(T)x\, :\, Q\in {\mathcal A}\}\subset X\setminus J(x)^{o},
\end{equation*}
which is a contradiction.
\end{proof}

\noindent\textbf{Proof of Theorem 4.1}  The set $\{ P(T)x:\, P\,\mbox{is a non-zero polynomial} \}$ is dense and connected. Assume that $J(x)\neq X$. So we can find a
non-zero polynomial $P$ such that $P(T)x$ $\in$ $\partial(J(x)^{o})$. This contradicts Lemma 4.5. \qed

\begin{corollary}
Let $T:X\rightarrow X$ be an operator. If there exists a cyclic vector $x\in X$ for $T$ such that $J(x)$ has non-empty interior then $T$ is hypercyclic.
\end{corollary}
\begin{proof}
The proof follows by combining Theorems 3.3 and 4.1.
\end{proof}

\begin{corollary}[Bourdon-Feldman's theorem]
Let $T:X\rightarrow X$ be an operator. If for some vector $x\in X$ the orbit $Orb(T,x)$ is somewhere dense then it is everywhere dense.
\end{corollary}
\begin{proof}
It is easy to see that $x$ is a cyclic vector for $T$. Since $Orb(T,x)$ is somewhere dense, it follows that $L(x)^o\neq \emptyset$. Note that $L(x)\subset J(x)$.
Hence Theorem 4.1 implies that $J(x)=X$. The set $\overline{Orb(T,x)}$ has non-empty interior so we can find a positive integer $l$ such that $T^{l}x\in
\overline{Orb(T,x)}^{o}$. Since $J(x)=X$ and $J(x)\subset J(T^{l}x)$ we arrive at $J(T^{l}x)=X$. So it is enough to prove that $\overline{Orb(T,x)}=J(T^{l}x)$. Let
$y\in J(T^{l}x)$. There exist a sequence $\{ x_{n}\} \subset X$ and a strictly increasing sequence of positive integers $\{ k_{n}\}$ such that $x_{n} \rightarrow
T^{l}x$ and $T^{k_{n} }x_{n} \rightarrow y$. Observing that $T^{l}x\in \overline{Orb(T,x)}^{o}$, without loss of generality we may assume that $x_{n}\in
\overline{Orb(T,x)}^{o}$ for every $n$. Moreover $\overline{Orb(T,x)}$ is $T$-invariant, hence $T^{k_{n}}x_{n}\in \overline{Orb(T,x)}$ for every $n$. Since
$T^{k_{n}}x_{n}\rightarrow y$ we conclude that $y\in\overline{Orb(T,x)}$.
\end{proof}

\begin{corollary}
Let $T:X\rightarrow X$ be an operator. Suppose there exist a vector $x\in X$ and a polynomial $P$ such that $P(T)x$ is a cyclic vector for $T$. If the set $J(x)$ has
non-empty interior then $T$ is hypercyclic.
\end{corollary}
\begin{proof}
Since $P(T)x$ is a cyclic vector for $T$ it is obvious that $x$ is a cyclic vector for $T$. Using the hypothesis that the set $J(x)$ has non-empty interior, Corollary
4.6 implies the desired result.
\end{proof}

\begin{remark}
The conclusion of Corollary 4.6 does not hold in general if $x$ is a cyclic vector for $T$ and $J(P(T)x)=X$ for some polynomial $P$. To see that, consider the space
$X=\mathbb{C}  \oplus l^2(\mathbb{N})$ and let $B:l^2(\mathbb{N})\rightarrow l^2(\mathbb{N})$ be the backward shift operator. Define the operator $T=2I\oplus 3B:
X\rightarrow X$, where $I$ denotes the identity operator acting on $\mathbb{C}$. Take any hypercyclic vector $y$ for $3B$ and define $x=1\oplus y$. Then $x$ is cyclic
for $T$ (in fact $x$ is supercyclic for $T$) and obviously $T$ is not hypercyclic. In fact it holds that $J(x)=\emptyset$. Consider the polynomial $P(z)=z-2$. Then
$P(T)x=0\oplus P(3B)y$. Since $y$ is hypercyclic for $3B$, by a classical result due to Bourdon \cite{Bourdon}, the vector $P(3B)x$ is hypercyclic for $3B$ as well.
Then using a similar argument as in Remark 3.6 we conclude that $J(P(T)x)=J(0\oplus P(3B)y)=X$. In particular, the above shows that, \textit{if $T$ is cyclic and
$J(x)=X$ for some vector $x\in X$ then $T$ is not hypercyclic in general}. On the other hand, we have the following.
\end{remark}

\begin{corollary}
Let $T:X\rightarrow X$ be an operator. Suppose $P$ is a non-zero polynomial such that $P(T)$ has dense range. If $x$ is a cyclic vector for $T$, $P(T)x\neq 0$ and
$J(P(T)x)^o \neq \emptyset $ then $T$ is hypercyclic.
\end{corollary}
\begin{proof}
Lemma 4.2 implies that $P(T)x$ is a cyclic vector for $T$. Since $J(P(T)x)^o$ $\neq$ $\emptyset$, Corollary 4.6 implies that $T$ is hypercyclic.
\end{proof}

\section{$J$-class operators}
\begin{definition}
An operator $T:X\rightarrow X $ will be called a \textit{$J$-class operator} provided there exists a non-zero vector $x\in X$ so that the extended limit set of $x$
under $T$ (see Definition 2.2) is the whole space, i.e. $J(x)=X$. In this case $x$ will be called a \textit{$J$-class vector} for $T$.
\end{definition}

The reason we exclude the extended limit set of the zero vector is to avoid certain trivialities, as for example the multiples of the identity operator acting on
finite or infinite dimensional spaces. To explain briefly, for any positive integer $n$ consider the operator $\lambda I: \mathbb{C}^{n}\to \mathbb{C}^{n}$, where
$\lambda $ is a complex number of modulus greater than $1$ and $I$ is the identity operator. It is then easy to check that $J_{\lambda I}(0)=X$ and $J_{\lambda
I}(x)\neq \mathbb{C}^{n}$ for every $x\in \mathbb{C}^{n} \setminus \{ 0\}$. However, the extended limit set of the zero vector plays an important role in checking
whether an operator $T:X\to X$ -acting on a Banach space $X$- supports non-zero vectors $x$ with $J_T(x)=X$, see Proposition 5.9. Let us also point out that from the
examples we presented in section 3, see Remark 3.6, it clearly follows that this new class of operators does not coincide with the class of hypercyclic operators.

Let us turn our attention to non-separable Banach spaces. Obviously a non-separable Banach space cannot support hypercyclic operators. However, it is known that
topologically transitive operators may exist in non-separable Banach spaces, see for instance \cite{BFPW}. On the other hand in \cite{BeKa}, Berm\'{u}dez and Kalton
showed that the non-separable Banach space $l^{\infty}(\mathbb{N})$ of bounded sequences over $\mathbb{C}$ does not support topologically transitive operators. Below
we prove that the Banach space $l^{\infty}(\mathbb{N})$ supports $J$-class operators.

\begin{proposition}
Let $B:l^{\infty}(\mathbb{N})\to l^{\infty}(\mathbb{N})$ be the backward shift where $l^{\infty}(\mathbb{N})$ is the Banach space of bounded sequences over
$\mathbb{C}$, endowed with the usual supremum norm. Then for every $|\lambda |>1$, $\lambda B$ is a $J$-class operator. In fact we have the following complete
characterization of the set of $J$-class vectors. For every $|\lambda |>1$ it holds that
\begin{equation*} \{ x\in l^{\infty}(\mathbb{N}): J_{\lambda
B}(x)=l^{\infty}(\mathbb{N}) \}=c_0(\mathbb{N}),
\end{equation*}
where $c_0(\mathbb{N})=\{ x=(x_n)_{n\in \mathbb{N}}\in l^{\infty}(\mathbb{N}): \lim_{n\to +\infty}x_n= 0 \}$.
\end{proposition}
\begin{proof}
Fix $|\lambda |>1$. Let us first show that if $x$ is a vector in $l^{\infty}(\mathbb{N})$ with finite support then $J_{\lambda B}(x)=l^{\infty}(\mathbb{N})$. For
simplicity let us assume that $x=e_1=(1,0,0,\ldots )$. Take any $y=(y_1,y_2,\ldots )\in l^{\infty}(\mathbb{N})$. Define $x_n=(1,0,\ldots
,0,\frac{y_1}{\lambda^n},\frac{y_2}{\lambda^n},\ldots )$ where $0$'s are taken up to the $n$-th coordinate. Obviously $x_n\in l^{\infty}(\mathbb{N})$ and it is
straightforward to check that $x_n\to e_1$ and   $(\lambda B)^nx_n=y$ for all $n$. Hence, $J_{\lambda B}(e_1)=l^{\infty}(\mathbb{N}) $. Since the closure of the set
consisting of all the vectors with finite support is $c_0(\mathbb{N})$, an application of Lemma 2.5 gives that $c_0(\mathbb{N})$ is contained in $\{ x\in
l^{\infty}(\mathbb{N}): J_{\lambda B}(x)=l^{\infty}(\mathbb{N}) \}$. It remains to show the converse implication. Suppose that $J_{\lambda
B}(x)=l^{\infty}(\mathbb{N})$ for some non-zero vector $x=(x_{1},x_{2},\ldots )\in l^{\infty}(\mathbb{N})$. Then there exist a sequence $ y_n =(y_{n1},y_{n2},\ldots
)$, $n=1,2,\ldots $ in $l^{\infty}(\mathbb{N})$ and a strictly increasing sequence of positive integers $\{ k_n \}$ such that $y_n \to x$ and $(\lambda B)^{k_n
}y_n\to 0$. Consider $\epsilon >0$. There exists a positive integer $n_0$ such that $\| y_n-x\| <\epsilon $ and  $\| (\lambda B)^{k_n}y_n\|=|\lambda
|^{k_n}\sup_{m\geq k_n+1}|y_{nm}|<\epsilon$ for every $n\geq n_0$. Hence for every $m\geq k_{n_0}+1$ and since $|\lambda |>1$ it holds that $|x_m|\leq \| y_{n_0}-x\|
+| y_{n_0m}|<2\epsilon $. The last implies that $x\in c_0(\mathbb{N})$ and this completes the proof.
\end{proof}

\begin{remark}
The previous proof actually yields that for every $|\lambda |>1$, $J_{\lambda B}(x)=l^{\infty}(\mathbb{N})$ if and only if $0\in J_{\lambda B}(x)$.
\end{remark}

Next we show that certain operators, such as positive, compact, hyponormal and operators acting on finite dimensional spaces cannot be $J$-class operators. It is well
known that the above mentioned classes of operators are disjoint from the class of hypercyclic operators, see \cite{Kit}, \cite{Bou}.

\begin{proposition}
(i) Let $X$ be an infinite dimensional separable Banach space and $T:X\rightarrow X$ be an operator. If $T$ is compact then it is not a $J$-class operator.

(ii) Let $H$ be an infinite dimensional separable Hilbert space and $T:H\rightarrow H$  be an operator. If $T$ is positive or hyponormal then it is not a $J$-class
operator.
\end{proposition}
\begin{proof}
Let us prove assertion (i). Suppose first that $T$ is compact. If $T$ is a $J$-class operator, there exists a non-zero vector $x\in X$ so that $J(x)=X$. It is clear
that there exists a bounded set $C\subset X$ such that the set $Orb(T,C) $ is dense in $X$. Then according to Proposition 4.4 in \cite{Fel1} no component of the
spectrum, $\sigma (T)$, of $T$ can be contained in the open unit disk. However, for compact operators the singleton $\{ 0\}$ is always a component of the spectrum and
this gives a contradiction.

We proceed with the proof of the second statement. Suppose now that $T$ is hyponormal. If $T$ is a $J$-class operator, there exists a non-zero vector $h\in H$ so that
$J(h)=H$. Therefore there exists a bounded set $C\subset H$ which is bounded away from zero (since $h\neq 0$) such that the set $Orb(T,C)$ is dense in $X$. The last
contradicts Theorem 5.10 in \cite{Fel1}. The case of a positive operator is an easy exercise and is left to the reader.
\end{proof}

Below we prove that any operator acting on a finite dimensional space cannot be $J$-class operator.

\begin{proposition}
Fix any positive integer $l$ and let $A:{\mathbb{C}}^l \rightarrow {\mathbb{C}}^l$ be a linear map. Then $A$ is
not a $J$-class operator. In fact $J(x)^o=\emptyset$ for every $x\in {\mathbb{C}}^l\setminus \{ 0\}$.
\end{proposition}
\begin{proof}
By the Jordan's canonical form theorem for $A$ we may assume that $A$ is a Jordan block with eigenvalue $\lambda\in\mathbb{C}$. Assume on the contrary that there
exists a non-zero vector $x\in {\mathbb{C}}^{l}$ with coordinates $z_{1},\ldots ,z_{l}$ such that $J(x)^o=\emptyset$. If $\{ x_{n}\}\in {\mathbb{C}}^{l}$ is such that
$x_{n}\rightarrow x$ and $z_{n1},\ldots ,z_{nl}$ be the corresponding coordinates to $x_{n}$ then the $m$-th coordinate of $A^{n}x_{n}$ equals to
\begin{equation*}
\sum_{k=0}^{l-m} \left(
\begin{array}{c}
  n \\
  k \\
\end{array}
\right) \lambda^{n-k} z_{n(m+k)}.
\end{equation*}

If $| \lambda | <1$ then $J(x)=\{ 0\}$. It remains to consider the case $| \lambda | \geq 1$. Suppose $z_{l}\neq 0$. Then, for every strictly increasing sequence of
positive integers $\{ k_{n}\}$ the possible limit points of the sequence $\{ \lambda^{k_n} z_{nl} \}$ are: either $\infty$ in case $| \lambda | >1$ or a subset of the
circumference $\{ z\in \mathbb{C} : |z|=|z_{l}|\}$ in case $| \lambda | =1$. This leads to a contradiction since $J(x)^{o}\neq\emptyset$. Therefore, the last
coordinate $z_{l}$ of the non-zero vector $x\in {\mathbb{C}}^{l}$ should be $0$. In case $| \lambda | =1$ and since $z_{l}=0$ the only limit point of $\{
\lambda^{k_n} z_{nl} \}$ is $0$ for every strictly increasing sequence of positive integers $\{ k_{n}\}$. So $J(x)^{o}\subset {\mathbb{C}}^{l-1}\times \{ 0\}$, a
contradiction. Assume now that $| \lambda | >1$. For the convenience of the reader we give the proof in the case $l=3$. Take $y=(y_{1},y_{2},y_{3})\in J(x)$. There
exist a strictly increasing sequence $\{ k_{n}\}$ of positive integers and a sequence $\{ x_n\}\subset {\mathbb{C}}^{3}$ such that
$x_{n}=(x_{n1},x_{n2},x_{n3})\rightarrow (z_{1},z_{2},0)=x$ and $A^{k_{n}}x_{n}\rightarrow y$. Let $y_n=(y_{n1},y_{n2},y_{n3})=A^{k_{n}}x_{n}$. Hence we have
\begin{equation*}
\begin{array}{l}
y_{n3}= \lambda^{k_n}x_{n1} + k_{n} \lambda^{k_{n}-1}x_{n2} + \frac{k_{n}(k_{n}-1)}{2}\lambda^{k_{n}-2}x_{n3}\\
y_{n2}= \lambda^{k_n}x_{n2} + k_{n} \lambda^{k_{n}-1}x_{n3} \\
y_{n1}= \lambda^{k_n}x_{n3}.
\end{array}
\end{equation*}
Since $y_{n3}= \lambda^{k_n}x_{n3}\rightarrow y_{3}$ then $k_{n}(k_{n}-1)x_{n3}\rightarrow 0$. From $y_{n2}\rightarrow y_{2}$ we get
$\frac{y_{n2}}{k_{n}}=\frac{\lambda^{k_n}}{k_{n}^{2}} \, k_{n}x_{n2} + \lambda^{k_{n}-1}x_{n3} \rightarrow 0$. Using the fact that $\lambda^{k_n}x_{n3}\rightarrow
y_{3}$ it follows that the sequence $\{ \frac{\lambda^{k_n}}{k_{n}^{2}} \, k_{n}x_{n2} \}$ converges to a finite complex number, hence $k_{n}x_{n2}\rightarrow 0$. The
last implies $x_{n2}\rightarrow 0$, therefore $z_{2}=0$. We have $ x_{n1}=\frac{y_{n3}}{\lambda^{k_n}}-\frac{1}{\lambda} k_{n}x_{n2} -\frac{1}{2} \lambda^{2}
k_{n}(k_{n}-1)x_{n3} $. Observing that each one term on the right hand side in the previous equality goes to $0$, since $y_{n3}\rightarrow y_{3}$, we arrive at
$z_{1}=0$. Therefore $x=0$ which is a contradiction.
\end{proof}

\begin{remark}
The previous result does not hold in general if we remove the hypothesis that $A$ is linear even if the dimension of the space is $1$. It is well known that the
function $f:(0,1)\rightarrow (0,1)$ with $f(x)=4x(1-x)$ is chaotic, see \cite{De}. Consider any homeomorphism $g:(0,1)\rightarrow \mathbb{R}$. Take
$h=gfg^{-1}:\mathbb{R}\rightarrow \mathbb{R}$. Then it is obvious that there is a $G_{\delta}$ and dense set of points with dense orbits in $\mathbb{R}$. Applying
Theorem 3.1 (observe that this corollary holds without the assumption of linearity for $T$) we get that $J(x)=\mathbb{R}$, for every $x\in\mathbb{R}$.
\end{remark}

It is well known, see \cite{HeKi}, that if $T$ is a hypercyclic and invertible operator, its inverse $T^{-1}$ is hypercyclic. On the other hand, as we show below, the
previously mentioned result fails for $J$-class operators.
\begin{proposition}
There exists an invertible $J$-class operator $T$ acting on a Banach space $X$ so that its inverse $T^{-1}$ is not a $J$-class operator.
\end{proposition}
\begin{proof}
Take any hypercyclic invertible operator $S$ acting on a Banach space $Y$ and consider the operator $T=\lambda I_{\mathbb{C}} \oplus S: \mathbb{C} \oplus Y\rightarrow
\mathbb{C} \oplus Y$, for any fixed complex number $\lambda $ with $|\lambda |>1$. Then, arguing as in Remark 3.6 it is easy to show that $T$ is a $J$-class operator.
However its inverse $T^{-1}={\lambda }^{-1}I_{\mathbb{C}} \oplus S^{-1}$ is not a $J$-class operator since $|{\lambda }^{-1}|<1$.
\end{proof}

Salas in \cite{Salas0} answering a question of D. Herrero constructed a hypercyclic operator $T$ on a Hilbert
space such that its adjoint $T^*$ is also hypercyclic but $T\oplus T^*$ is not hypercyclic. In fact the
following (unpublished) result of Deddens holds:  \textit{suppose $T$ is an operator, acting on a complex
Hilbert space, whose matrix with respect to some orthonormal basis, consists entirely of real entries. Then
$T\oplus T^*$ is not cyclic}. A proof of Deddens result can be found in the expository paper \cite{Sh}.
Recently, Montes and Shkarin, see \cite{MoSh}, extended Deddens' result to the general setting of Banach space
operators. Hence it is natural to ask if there exists an operator $T$ such that $T\oplus T^*$ is a $J$-class
operator. Below we show that this is not the case.

\begin{proposition}
Let $T$ be an operator acting on a Hilbert space $H$. Then $T\oplus T^*$ is not a $J$-class operator.
\end{proposition}
\begin{proof}
We argue by contradiction, so assume that $T\oplus T^*$ is a $J$-class operator. Hence there exist vectors $x,y\in H$ such that $J(x\oplus y)=H\oplus H$ and  $x\oplus
y\neq 0$.

\textit{Case I}: suppose that one of the vectors $x,y$ is zero. Without loss of generality assume $x=0$. Then there exist sequences $\{ x_n\}, \{y_n\} \subset H$ and
a strictly increasing sequence of positive integers $\{ k_n \}$ such that $x_n\rightarrow x=0$, $y_n\rightarrow y$, $T^{k_n}x_n\rightarrow y$ and
$T^{*k_n}y_n\rightarrow x=0$. Taking limits to the following equality $<T^{k_n}x_n,y_n>=<x_n,T^{*k_n}y_n>$ we get that $\| y\| =\| x\| =0$ and hence $y=0$. Therefore
$x\oplus y=0$, which yields a contradiction.

\textit{Case II}: suppose that $x\neq 0$ and $y\neq 0$. Let us show first that $J(\lambda x\oplus \mu y)=H\oplus H$ for every $\lambda ,\mu \in \mathbb{C} \setminus
\{ 0\}$. Indeed, fix $\lambda ,\mu \in \mathbb{C} \setminus \{ 0\}$. Take any $z,w \in H$. Since $J(x\oplus y)=H\oplus H$, there exist sequences $\{ x_n\}, \{y_n\}
\subset H$ and a strictly increasing sequence of positive integers $\{ k_n \}$ such that $x_n\rightarrow x$, $y_n\rightarrow y$, $T^{k_n}x_n\rightarrow {\lambda
}^{-1}z$ and $T^{*k_n}y_n\rightarrow {\mu }^{-1}w$. The last implies that $z\oplus w\in J(\lambda x\oplus \mu y)$, hence $J(\lambda x\oplus \mu y)=H\oplus H$. With no
loss of generality we may assume that $\| x\| \neq \| y\| $ (because if $\| x\| =\| y\| $, by multiplying with a suitable $\lambda \in \mathbb{C} \setminus \{ 0\}$ we
have $\| \lambda x\| \neq \| y\| $ and $J(\lambda x\oplus y)=H\oplus H$). Then we proceed as in Case I and arrive at a contradiction. The details are left to the
reader.
\end{proof}

Below we establish that, for a quite large class of operators, an operator $T$ is a $J$-class operator if and only if $J(0)=X$. What we need to assume is that there
exists at least one non-zero vector having ``regular" orbit under $T$.
\begin{proposition}
Let $T:X\rightarrow X$ be an operator on a Banach space $X$.
\begin{enumerate}
\item[(i)] For every positive integer $m$ it holds that $J_T(0)=J_{T^m}(0)$.

\item[(ii)] Suppose that $z$ is a non-zero periodic point for $T$. Then the following are equivalent.
\begin{enumerate}
\item[(1)] $T$ is a $J$-class operator;

\item[(2)] $J(0)=X$;

\item[(3)] $J(z)=X$.
\end{enumerate}

\item[(iii)] Suppose there exist a non-zero vector $z\in X$, a vector $w\in X$ and a sequence $\{ z_n\} \subset X$ such that $z_n\to z$ and $T^nz_n\to w$. Then the
following are equivalent.
\begin{enumerate}
\item[(1)] $T$ is a $J$-class operator;

\item[(2)] $J(0)=X$;

\item[(3)] $J(z)=X$.
\end{enumerate}
In particular, this statement holds for operators with non trivial kernel or for operators having at least one non-zero fixed point.
\end{enumerate}
\end{proposition}
\begin{proof}
Let us first show item (i). Fix any positive integer $m$ and let $y\in J_T(0)$. There exist a strictly increasing sequence of positive integers $\{ k_n\}$ and a
sequence $\{ x_n \}$ in $X$ such that $x_n\to 0$ and $T^{k_n}x_n\to y$. Then for every $n$ there exist non-negative integers $l_n, \rho_n $ with $\rho_n \in \{
0,1,\ldots ,m-1\}$ such that $k_n=l_nm+\rho_n$. Hence without loss of generality we may assume that there is $\rho \in \{ 0,1,\ldots ,m-1\}$ such that $
k_n=l_nm+\rho$ for every $n$. The last implies that $ T^{ml_n}(T^{\rho }x_n)\to y$ and $T^{\rho }x_n\to 0$ as $n\to \infty$. Hence $J_T(0)\subset J_{T^m}(0)$. The
converse inclusion is obvious. Let us show assertion (ii). That (1) implies (2) is an immediate consequence of Lemma 2.11. We shall prove that (2) gives (3). Suppose
that $N$ is the period of the periodic point $z$. Fix $w\in X$. Assertion (i) yields that $J_{T^N}(0)=X$. Hence there exist a strictly increasing sequence of positive
integers $\{ m_n\}$ and a sequence $\{ y_n \}$ in $X$ such that $ y_n\to 0$ and $T^{Nm_n}y_n\to w-z$. It follows that $ y_n+z\to z$ and $T^{Nm_n}(y_n+z) \to w$, from
which we conclude that $J_T(z)=X$. This proves assertion (ii). We proceed with the proof of assertion (iii). It only remains to show that (2) implies (3). Take any
$y\in X$. There exist a sequence $\{ x_n \} \subset X$ and a strictly increasing sequence $\{ k_n\}$ of positive integers such that $x_n \rightarrow 0$ and
$T^{k_n}x_n\rightarrow y-w$. Our hypothesis implies that $ x_n+z_{k_n}\rightarrow z$ and $T^{k_n}(x_n+z_{k_n})\rightarrow y$. Hence $y\in J(z)$.
\end{proof}

In the following proposition we provide a construction of $J$-class operators which are not hypercyclic.

\begin{proposition}
Let $X$ be a Banach space and let $Y$ be a separable Banach space. Consider an operator $S:X\rightarrow X$ so that $\sigma (S)\subset\{ \lambda : |\lambda |>1\}$. Let
also $T:Y\rightarrow Y$ be a hypercyclic operator. Then

\begin{enumerate}
\item[(i)]$S\oplus T:X\oplus Y\rightarrow X\oplus Y$ is a $J$-class operator but not a hypercyclic operator and

\item[(ii)] the set $\{ x\oplus y: x\in X,y\in Y\,\,\mbox{such that}\,\, J(x\oplus y)=X\oplus Y\}$ forms an infinite dimensional closed subspace of $X\oplus Y$ and in
particular
\end{enumerate}
\begin{equation*}
\{ x\oplus y: x\in X,y\in Y\,\,\mbox{such that}\,\, J(x\oplus y)=X\oplus Y\}=\{ 0\}\oplus Y.
\end{equation*}
\end{proposition}
\begin{proof}
We first prove assertion (i). That $S\oplus T$ is not a hypercyclic operator is an immediate consequence of the fact that $\sigma (S)\subset\{ \lambda : |\lambda
|>1\}$. Let us now prove that $S\oplus T$ is a $J$-class operator. Fix any hypercyclic vector $y\in Y$ for $T$. We shall show that $J(0\oplus y)=X\oplus Y$. Take
$x\in X$ and $w\in Y$. Since $\sigma (S)\subset\{ \lambda : |\lambda |>1\}$ it follows that $S$ is invertible and $ \sigma (S^{-1})\subset\{ \lambda : |\lambda
|<1\}$. Hence the spectral radius formula implies that $\| S^{-n}\|\rightarrow 0$. Therefore $S^{-n}x\rightarrow 0$. Since $y$ is hypercyclic for $T$ there exists a
strictly increasing sequence of positive integers $\{ k_{n}\}$ such that $T^{k_{n} }y \rightarrow w$. Observe now that $(S\oplus T)^{k_n}(S^{-k_{n}}x\oplus y)=x\oplus
T^{k_n}y\rightarrow x\oplus w $ and $S^{-k_{n}}x\oplus y\rightarrow 0\oplus y $. We proceed with the proof of (ii). Fix any hypercyclic vector $y\in Y$ for $T$. From
the proof of (i) we get $J(0\oplus y)=X\oplus Y$. Since for every positive integer $n$ the vector $T^{n}y$ is hypercyclic for $T$, by the same reasoning as above we
have that $J(0\oplus T^{n}y)=X\oplus Y$. Using Lemma 2.5 and that $y$ is hypercyclic for $T$ we conclude that $J(0\oplus w)=X\oplus Y$ for every $w\in Y$. To finish
the proof, it suffices to show that if $x\in X\setminus \{ 0\}$ then for every $w\in Y$, $J(x\oplus w)\neq X$. In particular we will show that $J(x\oplus
w)=\emptyset$. Suppose there exists $h\in J^{+}(x)=J(x)$ (see Definition 2.2). Propositions 2.9 and 2.10 imply that $x\in J^{-}(h)=L^{-}(h)$ (since $S^{-1}$ is power
bounded). On the other hand $\|S^{-n} \|\rightarrow 0$ and therefore $x\in L^{-}(h)=\{ 0\}$, which is a contradiction.
\end{proof}

We next provide some information on the spectrum of a $J$-class operator. Recall that if $T$ is hypercyclic then every component of the spectrum $\sigma (T)$
intersects the unit circle, see \cite{Kit}. Although the spectrum of a $J$-class operator intersects the unit circle $\partial D$, see Proposition \ref{spectrum1}
below, it may admits components not intersecting $\partial D$. For instance consider the $J$-class operator $2B\oplus 3I$, where $B$ is the backward shift on
$l^2(\mathbb{N})$ and $I$ is the identity operator on $\mathbb{C}$.

\begin{proposition}\label{spectralradius}
Let $T:X\rightarrow X$ be an operator on a complex Banach space $X$. If $r(T)<1$, where $r(T)$ denotes the spectral radius of $T$, or $\sigma (T)\subset \{ \lambda :
|\lambda |
>1 \} $ then $T$ is not a $J$-class operator.
\end{proposition}
\begin{proof}
If $r(T)<1$ then we have $\| T^n\| \to 0$. Hence $T$ is not a $J$-class operator. If $\sigma (T)\subset \{ \lambda : |\lambda | >1 \} $ the conclusion follows by the
proof of Proposition 5.10.
\end{proof}

\begin{proposition}\label{spectrum1}
Let $X$ be a complex Banach space. If $T:X\to X$ is a $J$-class operator, it holds that $\sigma (T) \cap \partial D \neq \emptyset$.
\end{proposition}
\begin{proof}
Assume, on the contrary, that $\sigma (T) \cap \partial D =\emptyset$. Then we have $\sigma (T)=\sigma_1\cup \sigma_2$ where $\sigma_1=\{ \lambda \in \mathbb{C}:
|\lambda |<1\}$ and $\sigma_2=\{ \lambda \in \mathbb{C}: |\lambda |>1\}$. If at least one of the sets $\sigma_1$, $\sigma_2$ is empty, we reach a contradiction
because of Proposition \ref{spectralradius}. Assume now that both $\sigma_1$, $\sigma_2$ are non-empty. Applying Riesz decomposition theorem, see \cite{RR}, there
exist invariant subspaces $X_1$, $X_2$ of $X$ under $T$ such that $X=X_1\oplus X_2$ and $\sigma (T_i)=\sigma_i$, $i=1,2$, where $T_i$ denotes the restriction of $T$
to $X_i$, $i=1,2$. It follows that $T=T_1\oplus T_2$ and since $T$ is $J$-class it is easy to show that at least one of $T_1$, $T_2$ is a $J$-class operator. By
Proposition \ref{spectralradius} we arrive again at a contradiction.
\end{proof}

\begin{proposition}
Let $T:l^{2}(\mathbb{N})\rightarrow l^{2}(\mathbb{N})$ be a unilateral backward weighted shift with positive weight sequence $\{ \alpha_n\}$ and consider a vector
$x=(x_1,x_2,\ldots )\in l^{2}(\mathbb{N})$. The following are equivalent.
\begin{enumerate}
\item[(i)] $T$ is hypercyclic;

\item[(ii)] $J(x)=l^{2}(\mathbb{N})$;

\item[(iii)] $J(x)^o\neq\emptyset$.
\end{enumerate}
\end{proposition}
\begin{proof}
It only remains to prove that (iii) implies (i). Suppose $J(x)^o\neq\emptyset$.  Then there exists a vector $y=(y_1,y_2,\ldots )\in J(x)$ such that $y_1\neq 0$. Hence
we may find a strictly increasing sequence $\{ k_n\}$ of positive integers and a sequence $\{ z_n \}$ in $l^{2}(\mathbb{N})$, $z_n=(z_{n1},z_{n2},\ldots )$, such that
$z_n\to x$ and $T^{k_n}z_n\to y$. We have
\begin{equation*}
|(T^{k_n}z_n)_1-y_1|=\left|\left( \prod_{i=1}^{k_n}\alpha_{i}\right) z_{n(k_n+1)}-y_1\right| \to 0.
\end{equation*}
Observe that $|z_{n(k_n+1)}|\leq |z_{n(k_n+1)}-x_{k_n+1}|+|x_{k_n+1}|\leq \| z_n-x\|+|x_{k_n+1}|$. The above
inequality implies $z_{n(k_n+1)}\to 0$ and since $y_1\neq 0$ we arrive at $\prod_{i=1}^{k_n}\alpha_{i}\to
+\infty $. By Salas' characterization of hypercyclic unilateral weighted shifts, see \cite{Salas2}, it follows
that $T$ is hypercyclic.
\end{proof}

\begin{remark}
We would also like to mention that (ii) implies (i) in the previous proposition, is an immediate consequence of
Proposition 5.3 in \cite{Fel1}. Let us stress that in case $T$ is a unilateral  backward weighted shift on
$l^{2}(\mathbb{N})$, the condition $J(0)=l^{2}(\mathbb{N})$ implies that $T$ is hypercyclic. For a
characterization of $J$-class unilateral weighted shifts on $l^{\infty}(\mathbb{N})$ in terms of their weight
sequence see \cite{CosMa}.
\end{remark}

\begin{proposition}
Let $T:l^{2}(\mathbb{Z})\rightarrow l^{2}(\mathbb{Z})$ be a bilateral backward weighted shift with positive weight sequence $\{ \alpha_n\}$ and consider a non-zero
vector $x=(x_n)_{n\in\mathbb{Z}}$ in $l^{2}(\mathbb{Z})$. The following are equivalent.
\begin{enumerate}
\item[(i)] $T$ is hypercyclic;

\item[(ii)] $J(x)=l^{2}(\mathbb{Z})$;

\item[(iii)] $J(x)^o\neq\emptyset$.
\end{enumerate}
\end{proposition}
\begin{proof}
It suffices to show that (iii) implies (i). In view of Salas' Theorem 2.1 in \cite{Salas2}, we shall prove that there exists a strictly increasing sequence $\{ k_n\}$
of positive integers such that for any integer $q$, $\prod_{i=1}^{k_n}\alpha_{i+q}\to +\infty$ and $ \prod_{i=0}^{k_n-1}\alpha_{q-i}\to 0$. Since $x$ is a non-zero
vector, there exists an integer $m$ such that $x_m\neq 0$. Without loss of generality we may assume that $m$ is positive. Suppose $J(x)^o\neq\emptyset$. Then there
exists a vector $y=(y_n)_{n\in\mathbb{Z}}$ in $l^{2}(\mathbb{Z})$ such that $y_1\neq 0$. Hence we may find a strictly increasing sequence $\{ k_n\}$ of positive
integers and a sequence $\{ z_n \}$ in $l^{2}(\mathbb{Z})$, $z_n=(z_{nl})_{l\in\mathbb{Z}}$, such that $z_n\to x$ and $T^{k_n}z_n\to y$. For simplicity reasons we
assume that $q=0$. Arguing as in the proof of Proposition 5.13 we get that $\prod_{i=1}^{k_n}\alpha_{i}\to +\infty $. On the other hand observe that
\begin{equation*}
|(T^{k_n}z_n)_{m-k_{n}}-y_{m-k_{n}}|=\left| \left( \prod_{i=0}^{m}\alpha_{i}\right) \left( \prod_{i=1}^{k_n-m+1}\alpha_{-i}\right) z_{nm}-y_{m-k_{n}}\right| \to 0.
\end{equation*}
Since $x_m\neq 0$ there exists a positive integer $n_{0}$ such that $ |z_{nm}|\geq \frac{|x_m|}{2}$ for every $n\geq n_{0}$. We also have $(T^{k_n}z_n)_{m-k_{n}}\to 0
$. The above imply that $\prod_{i=0}^{k_n-1}\alpha_{-i}\to 0 $.
\end{proof}

\section{Open problems}
\textbf{Problem 1}.

Let $T:X\to X$ be an operator on a Banach space $X$. Suppose there exists a vector $x\in X$ such that $J(x)^o\neq \emptyset$. Is it true that $J(x)=X$?\medskip

Ansari \cite{A1} and Bernal \cite{Bernal} gave a positive answer to Rolewicz' question if every separable and infinite dimensional Banach space supports a hypercyclic
operator. Observe that we showed that the non-separable Banach space $l^{\infty }(\mathbb{N})$ admits a $J$-class operator, while on the other hand Berm\'{u}dez and
Kalton \cite{BeKa} showed that $l^{\infty }(\mathbb{N})$ does not support topologically transitive operators. Hence it is natural to raise the following
question.\medskip

\textbf{Problem 2}.

Does every non-separable and infinite dimensional Banach space support a $J$-class operator? \medskip

D. Herrero in \cite{H2} established a spectral description of the closure of the set of hypercyclic operators acting on a Hilbert space. Below we ask a similar
question for $J$-class operators.\medskip

\textbf{Problem 3}.

Is there a spectral description of the closure of the set of $J$-class operators acting on a Hilbert space?\medskip

\textbf{Problem 4}.

Let $X$ be a separable Banach space and $T:X\to X$ be an operator. Suppose that $J(x)^o\neq \emptyset$ for every $x\in X$. Does it follow that $T$ is
hypercyclic?\medskip

Inspired by Grivaux's result that every operator on a complex Hilbert space can be written as a sum of two
hypercyclic operators, we consider the following. \medskip

\textbf{Problem 5}.

Is it true that any operator on $l^{\infty}(\mathbb{N})$ can be written as a sum of two $J$-class operators?

\end{document}